\documentclass[graybox]{svmult}

\usepackage{mathptmx}       
\usepackage{helvet}         
\usepackage{courier}        
\usepackage{type1cm}        
\usepackage{makeidx}         
\usepackage{graphicx}        
\usepackage{multicol}        
\usepackage[bottom]{footmisc}

\usepackage{amssymb}
\usepackage{graphicx}
\usepackage[T1]{fontenc}
\usepackage{hyperref}
\usepackage{multirow}
\usepackage{array}
\usepackage{stmaryrd}
\usepackage{breqn}

\newtheorem{observation}[theorem]{Observation}

\makeindex

\begin{document}

\title*{On higher congruences between cusp forms and Eisenstein series}
\author{Bartosz Naskr\k{e}cki}
\institute{Bartosz Naskr\k{e}cki \at Graduate School, Faculty of
Mathematics and Computer Science, Adam Mickiewicz University, Pozna\'{n}, Poland,
\email{bartnas@amu.edu.pl}}

\maketitle





\abstract{The paper contains a numerical study of congruences modulo prime powers between newforms and Eisenstein series at prime levels and with equal weights. We study the upper bound on the exponent of the congruence and formulate several observations based on the results of our computations.}

\section{Introduction}
\label{sec:1}
Let $p$ be a rational prime. For a newform $f\in\mathcal{S}_{k}(\Gamma_{0}(p))$, let $K_{f}=\mathbb{Q}(\{a_{n}(f)\}_{n\geq 0})$ be the field generated by the Fourier coefficients of the form $f$ and let $\mathcal{O}_{f}$ be its ring of integers.
Let $E_{k}$ denote the Eisenstein series of weight $k$ given by the $q$-expansion $-\frac{B_{k}}{2k}+\sum_{n=1}^{\infty}\left(\sum_{d\mid n}d^{k-1}\right)q^{n}$, where $B_{k}$ is the $k$-th Bernoulli number. We define the series $E_{k}^{(p)}$ by $E_{k}^{(p)}(\tau)=E_{k}(p\tau)$. From the theorem of Mazur \cite[Proposition 5.12, Proposition 9.6]{Mazur} we know that for $k=2$ and for any fixed prime $p\geq 11$ if we choose any prime $\ell\neq 2,3$ dividing the numerator of the zeroth coefficient of the Eisenstein series $E_{2}-pE_{2}^{(p)}$ of weight $2$, then there exists a newform $f$ in $\mathcal{S}_{2}(\Gamma_{0}(p))$ and a maximal ideal $\lambda\in\mathcal{O}_{f}$ above $\ell$ such that
\begin{equation}\label{Mazur_eq}
a_{r}(f)\equiv a_{r}(E_{2}-pE_{2}^{(p)})\textrm{ mod }\lambda
\end{equation}
for almost all primes $r$.

\medskip\noindent 
We study a generalization of the congruence (\ref{Mazur_eq}). Choose $E=E_{k}-p^{k-1}E_{k}^{(p)}$. Assume there exists a newform $f\in\mathcal{S}_{k}(\Gamma_{0}(p))$, a natural number $r\geq 1$ and a maximal ideal $\lambda\in\mathcal{O}_{f}$, such that
\begin{equation}\label{congruence}
a_{n}(E)\equiv a_{n}(f)\textrm{ mod }\lambda^{r}
\end{equation}
for all $n\geq 0$. Let $\ell$ be the rational prime below $\lambda$ and assume that $p\notin\lambda$. Then $\ell$ divides the numerator of $a_{0}(E)$. More precisely, \[r\leq ord_{\lambda}(\ell)v_{\ell}(\frac{-B_{k}}{2k}(1-p)),\]
where $ord_{\lambda}$ is the standard normalized $\lambda$-adic valuation on $K_{f}$ and $v_{\ell}$ is the $\ell$-adic valuation on $\mathbb{Q}$.
This is proved in Corollary \ref{corollary:bounds} and Lemma \ref{lemma:bound}. In the proof, we use the explicit description of $a_{p}(f)$ for a newform $f\in\mathcal{S}_{k}(\Gamma_{0}(p))$, cf. \cite[Theorem 3]{Atkin}. In general, the maximal exponent $r$ of the congruence (\ref{congruence}) can be equal to 
\[m:=e\cdot v_{\ell}(\frac{-B_{k}}{2k}(1-p)),\]
where $e=ord_{\lambda}(\ell)$. From now on $r$ always denotes the maximal exponent of the congruence.
\begin{proposition}
There exists a prime $p$, a positive even integer $k$ and a newform $f\in\mathcal{S}_{k}(\Gamma_{0}(p))$ such that the congruence (\ref{congruence}) holds for all $n\geq 0$ and for some $r=m>1$.
\end{proposition}
\begin{proof}
We present an explicit example in Section \ref{subsection:r_eq_m_gt_1_case}.
\end{proof}

\begin{proposition}
There exists a prime $p$, a positive even integer $k$ and a newform $f\in\mathcal{S}_{k}(\Gamma_{0}(p))$ such that the congruence (\ref{congruence}) holds for all $n\geq 0$ with $m>r$.
\end{proposition}
\begin{proof}
We present an explicit example in Section \ref{subsection:m_gt_r_gt_1_e_eq_1_case} ( for $e=1$ and $r>1$), in Section \ref{subsection:m_gt_e_gt_r_e_gt_1_case} (for $e>1$ and $r=1$) and in Section \ref{subsection:m_gt_e_eq_r_e_gt_1_case} (for $e>1$ and $r=e>1$), respectively.
\end{proof}

\begin{observation}\label{observation:bound}
Let $k\leq 22$ be an even positive integer and $p$ be a prime, which is bounded with respect to $k$ as indicated in the table below.
\[\begin{array}{|c|c|c|c|c|c|c|c|c|c|c|c|}
    \hline
    k          & 2     & 4 & 6    & 8  & 10 & 12  & 14 & 16 & 18 & 20 & 22 \\
    \hline
    p\leq & 1789 & 397 & 229 & 193 & 109 & 113 & 97 & 71 & 67 & 67 & 59 \\
    \hline
\end{array}\]
Let $\ell>2$ be a prime such that $v_{\ell}(a_{0}(E_{k}-p^{k-1}E_{k}^{(p)}))>0$. Let $f\in \mathcal{S}_{k}(\Gamma_{0}(p))$ be a newform and $\lambda\in\mathcal{O}_{f}$ be a prime ideal above $\ell$ which is ramified, i.e. $ord_{\lambda}(\ell)=e>1$. If
\[a_{n}(f)\equiv a_{n}(E_{k}-p^{k-1}E_{k}^{(p)})\textrm{ mod }\lambda^{r}\]
for all $n\geq 0$, then $r\leq e$ for every computed case (we found all possible examples in the range described in Table \ref{tab:Levels_full}, Section \ref{Section_numerical} and some further examples in the range described in Table \ref{tab:Levels}, Section \ref{Section_numerical}). We present nontrivial examples of such congruences in Table \ref{tab:RamifiedHighPowers}, Section \ref{Section_numerical}, marking them with boldface.
\end{observation}

\noindent
For $\ell=2$ and $e>1$ the upper bound $r\leq e$ does not hold.
\begin{proposition}
Let $\ell=2$. There exists a prime level $p$ and a newform $f\in S_{2}(\Gamma_{0}(p))$ such that
\[a_{n}(f)\equiv a_{n}(E_{2}-p E_{2}^{(p)})\textrm{ mod } \lambda^{r}\]
for all $n\geq 0$ and a prime ideal $\lambda\in\mathcal{O}_{f}$ above $2$, such that $1<ord_{\lambda}(2)<r$.
\end{proposition}
\begin{proof}
Put $p=257$. There is a newform $f$ in $\mathcal{S}_{2}(\Gamma_{0}(257))$, which satisfies the congruence (\ref{congruence}) with $r=5$ and $ord_{\lambda}(2)=2$ for a suitable $\lambda\in \mathcal{O}_{f}$, cf. Table \ref{tab:RamifiedHighPowers}, Section \ref{Section_numerical}.
\end{proof}

\noindent
The case of congruences between rational newforms and Eisenstein series in weight $k=2$ is easy to handle as the next result indicates.
\begin{proposition}
Let $f\in S_{2}(\Gamma_{0}(p))$ be a newform with rational coefficients (associated to an elliptic curve over $\mathbb{Q}$). Let $\ell$ be a prime such that $v_{\ell}(a_{0}(E_{2}-pE_{2}^{(p)}))>0$ and
\[a_{n}(f)\equiv a_{n}(E_{2}-pE_{2}^{(p)})\textrm{ mod }\ell^{r}\]
for all $n\geq 0$ and some $r>0$. Then one of the following holds
\begin{enumerate}
\item $\ell=3$, $r=1$ and $p=19$ or $p=37$,
\item $\ell=5$, $r=1$ and $p=11$,
\item $\ell=2$, $r=1$ and $p=17$,
\item $\ell=2$, $r=1$ and $p=u^2+64$ for some integer $u$.
\end{enumerate}
\end{proposition}
\begin{proof}
The proof is given in Section \ref{subsection:congruences_over_Q}.
\end{proof}

\noindent
The number of different newforms congruent to the same Eisenstein series may vary.

\begin{proposition}
There exists a prime $p>2$ and two newforms $f_{1},f_{2}\in\mathcal{S}_{2}(\Gamma_{0}(p))$ (which are not Galois conjugated), such that
\[ a_{n}(f_{1}) \equiv a_{n}(E_{2}-pE_{2}^{(p)})\textrm{ mod }\lambda_{1}^{r_{1}},\]
\[a_{n}(f_{2})\equiv a_{n}(E_{2}-pE_{2}^{(p)})\textrm{ mod }\lambda_{2}^{r_{2}},\]
for all $n\geq 0$, prime ideals $\lambda_{1}\in\mathcal{O}_{f_{1}}$,$\lambda_{2}\in\mathcal{O}_{f_{2}}$ of equal residue field characteristic and some $r_{1},r_{2}>0$.
\end{proposition}
\begin{proof}
Consider a prime level $p=353$. The space $\mathcal{S}_{2}(\Gamma_{0}(353))$ has dimension $29$ and it has $4$ different Galois conjugacy classes of newforms. With respect to the internal MAGMA numbering, the first, the second and the fourth Galois conjugacy class contains a newform congruent to the Eisenstein series $E_{2}-pE_{2}^{(p)}$ modulo a prime above $\ell=2$. For $\ell=3$, we can take $p=487$ and find newforms congruent to the Eisenstein series in Galois conjugacy classes with numbers $2$ and $4$, cf. Section \ref{subsection:m_gt_r_gt_1_e_eq_1_case} and Table \ref{tab:UnramifiedHigherPowers}, Section \ref{Section_numerical}.
\end{proof}

\begin{remark}
Mazur suggests that there may be infinitely many such examples for $\ell=2$, cf. \cite[II.19]{Mazur}. 
\end{remark}

\noindent
The majority of our examples satisfy the property that the residue field $\mathcal{O}_{f}/\lambda$ is a field with $\ell$ elements. However, this is not always the case, cf. Section \ref{subsection:large_residue_field}. Also the ring of integers $\mathcal{O}_{f}$ is not always equal to $\mathbb{Z}[\theta]$, where $\theta$ is an algebraic integer such that $K_{f}=\mathbb{Q}(\theta)$. It can happen also that $[\mathcal{O}_{f}:\mathbb{Z}[\theta]]$ is divisible by $\ell$, i.e. $\mathbb{Z}[\theta]$ is not $\ell$-maximal in $\mathcal{O}_{f}$, cf. Section \ref{subsection:Equation_order_is_not_always}.

\medskip\noindent 
Our motivation to study the congruences (\ref{congruence}) comes from the question posed in \cite[Paragraph 4.4]{Wiese}. The authors study congruences between a newform $f\in\mathcal{S}_{2}(\Gamma_{0}(p))$ and an Eisenstein series $E\in\mathcal{E}_{2}(\Gamma_{0}(p))$
\[a_{n}(f)\equiv a_{n}(E)\,\mathop{mod}\,\lambda^{r},\]
which hold for almost all prime indices $n$. Our approach is more restrictive, since in the congruences we take into account all indices $n\geq 0$. 

\medskip\noindent
The computations done in Section \ref{subsection:m_gt_e_eq_r_e_gt_1_case} and Section \ref{subsection:large_residue_field} suggest that the equality holds in \cite[Question 4.1]{Wiese}.

\medskip\noindent
In the first case, where $k=2$ and $p=919$ we obtain two distinct Galois conjugate newforms which are congruent to the Eisenstein series modulo $3$ (in the sense of \cite[Definition 2.2]{Wiese}), while the numerator of $\frac{p-1}{12}$ is divisible by $9$. In the second case ($k=2, p=401$), we find a pair of distinct Galois conjugate newforms congruent to the Eisenstein series modulo $5$ (in the sense of \cite[Definition 2.2]{Wiese}), while the numerator of $\frac{p-1}{12}$ is divisible by $25$.


\medskip\noindent
It would be desirable to extend our computations to take into account the situation when the cusp form and the Eisenstein series have different weights. Such a computational and theoretical study was done for two cusp forms in \cite{Kiming}. Unfortunately, we cannot apply directly the results of the paper \cite{Kiming} to our situation.

\medskip\noindent
The proposed upper bound for $r$ recorded in Observation \ref{observation:bound} might be linked to the behavior of the inertia group at $p$ of the residually reducible Galois representation attached to the newform $f$. Such a condition is given in \cite[Theorem 2]{Dieulefait}, where the authors deal with congruences between two cusp forms.


\medskip\noindent
\textit{Contents of the paper:} In Section \ref{Section_notation} we introduce basic notation and describe Hecke algebras and eigenforms. Next, in Section \ref{Section_algorithm} we describe the upper bound for the exponent of congruences between cuspidal eigenforms and Eisenstein series.

\medskip\noindent
In Section \ref{Subsection_orders} for the convenience of the reader we collect basic facts of the theory of $\ell$-maximal orders which is an important ingredient of our algorithm. These facts are crucial for several improvements of the algorithm speed. 

\medskip\noindent
Section \ref{subsection:Sketch_algorithm} contains a pseudo code description of the main algorithm  which was implemented in MAGMA \cite{Magma}. The source code is available on request or online, cf. \cite{Bartek}.

\medskip\noindent
Section \ref{Section_numerical} is devoted to presentation of the numerical data and proofs of propositions stated above. We discuss several explicit examples and the numerical data collected in the tables.

\section{Notation and definitions}\label{Section_notation}
Let $p$ be a prime number and $k$ a positive even integer. The space $\mathcal{M}_{k}(\Gamma_{0}(p))$ of holomorphic modular forms of weight $k$ and level $\Gamma_{0}(p)$ splits over $\mathbb{C}$ into a direct sum
\[\mathcal{M}_{k}(\Gamma_{0}(p))=\mathcal{E}_{k}(\Gamma_{0}(p))\oplus\mathcal{S}_{k}(\Gamma_{0}(p))\]
of the Eisenstein part and the space of cuspidal modular forms (cf.\cite[Paragraph 5.11]{Diamond}).
\medskip\noindent
From dimension formulas for modular forms we have
\[\dim_{\mathbb{C}}(\mathcal{E}_{k}(\Gamma_{0}(p))=\left\{\begin{array}{ccc}
1,& & k = 2\\
2,& & k\geq 4.
\end{array}\right.\]
Let $\sigma_{r}(n)=\sum_{d\mid n}d^r$ and $q=e^{2\pi i \tau}$, where $\tau$ lies on the complex upper half-plane $\mathcal{H}$. We define
\[E_{k}(\tau)=-\frac{B_{k}}{2k}+\sum_{n=1}^{\infty}\sigma_{k-1}(n)q^n.\]
The sequence of Bernoulli numbers $\{B_{m}\}_{m\in\mathbb{N}}$ is defined as usual by the series $\sum_{m=0}^{\infty}{B_{m}}t^{m}=\frac{t}{e^{t}-1}$.
Explicitly, in $\mathcal{E}_{2}(\Gamma_{0}(p))$ we have a generator 
\[E_{2}(\tau)-pE_{2}(p\tau)=\frac{p-1}{24}+\sum_{n=1}^{\infty}\sigma_{1}(n)q^n-p\sum_{n=1}^{\infty}\sigma_{1}(n)q^{pn}.\]
The space $\mathcal{E}_{k}(\Gamma_{0}(p))$ is generated by $E_{k}(\tau)$ and $E_{k}(p\tau)$ when $k\geq 4$.

\medskip\noindent
The space of modular forms $\mathcal{M}_{k}(\Gamma_{0}(p))$ carries a natural action of a commutative $\mathbb{C}$-algebra $\mathbb{T}$ generated by the Hecke operators, cf.\cite[Proposition 5.2.1]{Diamond}. The algebra is generated by two types of operators. The first type is defined for the primes $\ell\neq p$ by the formula 
\[T_{\ell}(f)=\sum_{n=0}^{\infty}a_{n\ell}(f)q^n+\ell^{k-1}\sum_{n=0}^{\infty}a_{n}(f)q^{n\ell},\]
where $f\in \mathcal{M}_{k}(\Gamma_{0}(p))$ and $a_{n}(f)$ denotes the $n$-th Fourier coefficient of the form $f$ at infinity. 
For $\ell=p$ there is a single operator
\[T_{p}(f)=\sum_{n=0}^{\infty}a_{np}(f)q^{n}.\]

\noindent
We denote by $\mathbb{T}$ the algebra $\mathbb{C}[\{T_{r}\}]$ generated by the Hecke operators $T_{r}$, where $r$ runs through all rational primes. The action of Hecke algebra $\mathbb{T}$ on the space $\mathcal{M}_{k}(\Gamma_{0}(p))=\mathcal{E}_{k}(\Gamma_{0}(p))\oplus\mathcal{S}_{k}(\Gamma_{0}(p))$ preserves the direct sum splitting into Eisenstein and cuspidal parts. For $k=2$ since $\dim\mathcal{E}_{2}(\Gamma_{0}(p))=1$, the series $E_{2}-pE_{2}^{(p)}$ is the unique normalized eigenform. 

\medskip\noindent
For $k\geq 4$ the dimension of the space $\mathcal{E}_{k}(\Gamma_{0}(p))$ is equal to two.
We have a basis of the space consisting of normalized eigenforms
\[E_{k}-p^{k-1}E_{k}^{(p)},\quad E_{k}-E_{k}^{(p)}.\]


\section{Bounds on congruences}\label{Section_algorithm} 
Let $k$ be an even positive integer and $p$ be a rational prime. We want to find congruences between Eisenstein series $E_{k}-p^{k-1}E_{k}^{(p)}$ and cuspidal newforms in the space $\mathcal{M}_{k}(\Gamma_{0}(p))$.  
Let $f$ be a newform in $\mathcal{S}_{k}(\Gamma_{0}(p))$. Assume there exists a prime ideal $\lambda$ in $\mathcal{O}_{f}$ and a natural number $r$ such that
\begin{equation}\label{Eq:condition1}
a_{n}(E_{k}-p^{k-1}E_{k}^{(p)})\equiv a_{n}(f)\textrm{ mod }\lambda^{r}.
\end{equation}
The bound on $r$ depends on the $q$-expansion of the Eisenstein series at cuspidal points of the modular curve $X_{0}(p)$.
The modular curve $X_{0}(p)$ has two cusps, $0$ and $\infty$. Hence for any modular form $f\in\mathcal{M}_{k}(\Gamma_{0}(p))$ we have $q$-expansions at $\infty$ and $0$. We compute expansions for $f$ and $f\mid_{k}\left(\begin{array}{cc}
0 & -1\\
1 & 0
\end{array}\right)$. 
We denote by $\mu(f)$ the zeroth coefficient of the $q$-expansion of the form $f$ at $0$. Equivalently, this is the zeroth coefficient of the $q$-expansion of the form $f\mid_{k}\left(\begin{array}{cc}
0 & -1\\
1 & 0
\end{array}\right)$ at $\infty$. 

\begin{lemma}\label{lemma:bound}
Let $p$ be a prime number, $k\geq 2$ be an even integer and $f\in \mathcal{S}_{k}(\Gamma_{0}(p))$ be a newform.
Let $\lambda$ be a prime ideal in $\mathcal{O}_{f}$ such that $p\notin\lambda$ and let $r\geq 1$ be a natural number. Let $E$ denote the Eisenstein series $E_{k}-p^{k-1}E_{k}^{(p)}$. Suppose we have a congruence
\begin{equation}\label{equation:congruence}
a_{n}(f)\equiv a_{n}(E)\textrm{ mod }\lambda^{r}
\end{equation}
for all $n\geq 0$. Then $\mu(E)\equiv 0\textrm{ mod }\lambda^{r}$. Hence the form $E$ is cuspidal modulo $\lambda^{r}$. 
\end{lemma}

\begin{proof}
For any even integer $k\geq 2$ the formula
\[(E_{k}-p^{k-1}E_{2}^{(p)})\mid_{k}\left(\begin{array}{cc}
0 & -1\\
1 & 0
\end{array}\right)=E_{k}-\frac{1}{p}E_{k}^{(1/p)}\]
holds. It follows that $\mu(E)$ equals $-\frac{B_{k}}{2k}\left(1-\frac{1}{p}\right)$. By \cite[Theorem 3]{Atkin} we have that
\[a_{p}(f)=-\epsilon_{p}p^{k/2-1}\]
for some $\epsilon_{p}\in\{-1,1\}$. On the other side, $a_{p}(E)=1$, hence $-\epsilon_{p}p^{k/2-1}\equiv 1\,\mathop{mod}\,\lambda^{r}$ by (\ref{equation:congruence}). Thence, we obtain the congruence
\begin{equation}\label{equation:congruence_p}
1-p^{k-2}\equiv 0\,\mathop{mod}\,\lambda^{r}.
\end{equation}
Observe that $\mu(E)=-\frac{B_{k}}{2k}(1-p^{k-2})+\left(-\frac{B_{k}}{2k}(1-p^{k-1})\right)\left(-\frac{1}{p}\right)$.
Since $f$ is a cuspform, the assumption (\ref{equation:congruence}) implies that $a_{0}(E)\equiv 0\,\mathop{mod}\,\lambda^{r}$, hence $-\frac{B_{k}}{2k}(1-p^{k-1})\equiv 0\,\mathop{mod}\,\lambda^{r}$. The last congruence and the congruence (\ref{equation:congruence_p}) imply that $\mu(E)\equiv 0\,\mathop{mod}\,\lambda^{r}$.

\end{proof}

\begin{corollary}\label{corollary:bounds}
Let $p$ and $\ell$ be two distinct rational primes. Suppose that we have two integers $k\geq 2$, $r\geq 1$ and $k$ is even. Let $f$ be a newform in $\mathcal{S}_{k}(\Gamma_{0}(p))$. Suppose that for $E=E_{k}-p^{k-1}E_{k}^{(p)}$ we have congruences
\[a_{n}(f)\equiv a_{n}(E)\,\mathop{mod}\,\lambda^{r}\]
for all $n\geq 0$ and some prime ideal $\lambda$ in $\mathcal{O}_{f}$ dividing $\ell$. 
There is an upper bound
\[r\leq\textrm{ord}_{\lambda}(\ell)\cdot v_{\ell}\left(-\frac{B_{k}}{2k}(1-p)\right),\]
where $v_{\ell}$ denotes the $\ell$-adic valuation on $\mathbb{Q}$. 
\end{corollary}
\begin{proof}
By Lemma \ref{lemma:bound} we have $\mu(E)\equiv 0\,\mathop{mod}\,\lambda^{r}$, hence $-\frac{B_{k}}{2k}\left(1-\frac{1}{p}\right)\equiv 0\,\mathop{mod}\,\lambda^{r}$. Since $p$ is invertible modulo $\lambda$, then
\[-\frac{B_{k}}{2k}\left(1-p\right)\equiv 0\,\mathop{mod}\,\lambda^{r}.\]
The exponent $r$ satisfies the inequality $r\leq \mathop{min}(\textrm{ord}_{\lambda}(a_{0}(E)),\textrm{ord}_{\lambda}(\mu(E)))$. But we have for $k\geq 2$ \[\textrm{ord}_{\lambda}\left(-\frac{B_{k}}{2k}(1-p^{k-1})\right)\geq \textrm{ord}_{\lambda}\left(-\frac{B_{k}}{2k}(1-p)\right),\]
hence 
\[r\leq\textrm{ord}_{\lambda}(\ell)\cdot v_{\ell}\left(-\frac{B_{k}}{2k}(1-p)\right).\]
\end{proof}
\begin{remark}\label{remark:cases} 
Let $p\in\lambda$ or equivalently $p=\ell$. For $k>2$ the congruence (\ref{equation:congruence_p}) (which holds either when $p=\ell$ or when $p\neq \ell$) implies that $1-\ell^{k-2}\in\lambda^{r}$, hence $1-\ell^{k-2}\in\lambda$, hence $1\in\lambda$, which leads to a contradiction. If $k=2$, we observe that $a_{0}(E_{2}-pE_{2}^{(p)})\equiv 0\,\mathop{mod}\,\lambda^{r}$. Hence $\textrm{ord}_{\lambda}(-\frac{1}{24}(1-p))\geq r\geq 1$, so $1-p\in\lambda$ and $1\in\lambda$, since $\ell=p$, which is impossible.
\end{remark}

\medskip
\noindent
In the computations below we use a straightforward generalization of the well-known theorem of Sturm \cite{Sturm}, suitable for our purposes. A similar theorem in more general form is proved in \cite[Proposition 1]{Kiming}.

\begin{theorem}\label{Sturm_gen}
Let $p$ be a rational prime and $k\geq 2$ be an even integer. Let $f\in\mathcal{S}_{k}(\Gamma_{0}(p))$ be a normalized eigenform. Let $\ell$ be a rational prime dividing the numerator of $-\frac{B_{k}}{2k}(1-p^{k-1})$. Suppose we have a positive integer $r$ and a nonzero prime ideal $\lambda$ in $\mathcal{O}_{f}$, containing $\ell$, such that for all $n\leq \frac{k(p+1)}{12}$ there is a congruence
\[a_{n}(f)\equiv a_{n}(E_{k}-p^{k-1}E_{k}^{(p)})\,\mathop{mod}\,\lambda^{r}.\]
Then the congruence holds for all $n\geq 0$.
\end{theorem}
\begin{proof}
Let us denote by $B$ the number $\frac{k(p+1)}{12}$ and by $E$ the form $E_{k}-p^{k-1}E_{k}^{(p)}$. Let $m$ denote the denominator of $-\frac{B_{k}}{2k}(1-p^{k-1})$. Observe that $\ell\nmid m$. Fourier coefficients of the form  $mE$ are rational integers. Coefficients of the form $f$ lie in $\mathcal{O}_{f}$. If $r=1$, then we know that for $n\leq B$
\[a_{n}(mf)\equiv a_{n}(mE)\,\mathop{mod}\,\lambda,\]
hence by the theorem of Sturm (cf. \cite[Theorem 9.18]{Stein_book}) we get the congruence for all $n\geq 0$. But $m\notin\lambda$, hence
\[a_{n}(f)\equiv a_{n}(E)\,\mathop{mod}\,\lambda\]
for all $n\geq 0$. If $r>1$, then we proceed by induction. Assume first that the statement is true for $r-1$. Suppose that $a_{n}(f)\equiv a_{n}(E)\,\mathop{mod}\,\lambda^{r}$ for $n\leq B$. In particular, we get $a_{n}(f)\equiv a_{n}(E)\,\mathop{mod}\,\lambda^{r-1}$ for all $n\geq 0$ by the induction hypothesis. Choose any algebraic integer $\pi\in\lambda\setminus\lambda^{2}$. Then the function $\frac{1}{\pi^{r-1}}(f-E)$ is a modular form in $\mathcal{M}_{k}(\Gamma_{0}(p))$ with Fourier coefficients lying in the localization $(\mathcal{O}_{f})_{\lambda}$ of the ring $\mathcal{O}_{f}$ at the prime ideal $\lambda$. By a theorem of Shimura (cf. \cite[Theorem 3.52]{Shimura_book}) the space $\mathcal{S}_{k}(\Gamma_{0}(p))$ has a basis consisting of forms with with Fourier coefficients in $\mathbb{Z}$. The same is true for $\mathcal{M}_{k}(\Gamma_{0}(p))$. Hence, there exists an algebraic integer $\alpha\in\mathcal{O}_{f}\setminus\lambda$ such that $\frac{\alpha}{\pi^{r-1}}(f-E)$ has the Fourier expansion in $\mathcal{O}_{f}$. Moreover, for all $n\leq B$ the congruence 
\[a_{n}\left(\frac{\alpha}{\pi^{r-1}}(f-E)\right)\equiv 0\,\mathop{mod}\,\lambda\]
holds. By the Sturm theorem, it is true for all $n\geq 0$. This implies
\[a_{n}(\alpha(f-E))\equiv 0\,\mathop{mod}\,\lambda^{r}\]
for all $n\geq 0$. Since $\alpha\notin\lambda$, the induction step holds true and the theorem is proved.
\end{proof}

\section{Sketch of the algorithm}
\subsection{Orders in number fields}\label{Subsection_orders}
Fix a rational prime $\ell$. We briefly recall the concept of an $\ell$-maximal order. We fix an algebraic integer $\theta$ and the number field $K=\mathbb{Q}(\theta)$. By $\mathcal{O}_{K}$ we denote the ring of algebraic integers in $K$. 

\medskip\noindent
An order $\mathcal{O}$ in $K$ is $\ell$-\textit{maximal} if $\ell\nmid[\mathcal{O}_{K}:\mathcal{O}]$. It is always possible to construct an $\ell$-maximal order $\mathcal{O}$ containing the equation order $\mathbb{Z}[\theta]$, cf. \cite[ Theorem 6.1.3]{Cohen}. 

\medskip\noindent
Every prime ideal $\mathfrak{L}$ in $\mathcal{O}$ lying over the prime $\ell$ is invertible in $\mathcal{O}$, which follows from \cite[Proposition 4.8.15]{Cohen}, \cite[Theorem 6.1.3]{Cohen} and \cite[Proposition 6.1.2]{Cohen}. 

\medskip\noindent
By \cite[Theorem 11.4]{Matsumura_book} the localization $\mathcal{O}_{\mathfrak{L}}$ of the ring $\mathcal{O}$ at the prime ideal $\mathfrak{L}$ is a discrete valuation ring which is equal to $(\mathcal{O}_{K})_{\mathfrak{L}\mathcal{O}_{K}}$ by \cite[Proposition 12.10]{Neukirch_book} (from which it also follows that $\mathfrak{L}\mathcal{O}_{K}$ is a unique prime ideal in $\mathcal{O}_{K}$ above $\mathfrak{L}$). 

\medskip\noindent
We define a valuation on elements of an $\ell$-maximal order with respect to any prime ideal over $\ell$. For a nonzero prime ideal $\mathfrak{L}\in\textrm{ Spec }\mathcal{O}$ over $\ell$, let $\mathcal{L} = \mathfrak{L}\mathcal{O}_{K}$ be the corresponding prime ideal in $\mathcal{O}_{K}$. Any element $x\in\mathcal{O}$ can be written as $x=u_{1}\pi^{r}=u_{2}\Pi^{r}$, for $u_{1}\in\mathcal{O}_{\mathfrak{L}}^{\times}, u_{2}\in(\mathcal{O}_{K})_{\mathcal{L}}^{\times}$ and uniformizers $\pi$ and $\Pi$ in $(\mathcal{O})_{\mathfrak{L}}$ and $(\mathcal{O}_{K})_{\mathcal{L}}$, respectively. 

\noindent
The common exponent of uniformizers will be denoted by $ord_{\mathfrak{L}}(x):=r$.
The definition extends to the fraction field $K=Frac(\mathcal{O})$ as $ord_{\mathfrak{L}}\left(\frac{x}{y}\right)=ord_{\mathfrak{L}}(x)-ord_{\mathfrak{L}}(y)$. The following equivalence holds for any $x\in\mathcal{O}\subset\mathcal{O}_{K}$
\[ord_{\mathfrak{L}}(x)\geq r\quad\Leftrightarrow\quad x\equiv 0\textrm{ mod } \mathcal{L}^{r}.\]

\medskip\noindent 
In the algorithm presented in Section \ref{subsection:Sketch_algorithm} we use the last equivalence of orders. It is also crucial for the algorithm that the computation of an $\ell$-maximal order is more efficient than computation of the whole ring of algebraic integers, which involves factorization of discriminants. By the result of Chistov (cf. \cite[Theorem 1.3]{Buchmann_Lenstra}), computation of the ring of algebraic numbers in the number field $K$ is polynomially equivalent to finding the largest square-free divisor of the field discriminant. The latter problem has not found to date a satisfactory solution, better than just factorizing the whole integer. On the other hand, computation of an $\ell$-maximal order is straightforward and quick (cf. \cite[Chapter 6]{Cohen}). Computation of prime ideals above $\ell$ in an $\ell$-maximal order is equally fast, cf. \cite[Chapter 6.2]{Cohen}. In our computations we often exploit this feature of $\ell$-maximal orders.

\subsection{Algorithm}\label{subsection:Sketch_algorithm}
\noindent
Input: $(p,k)\in\mathbb{Z}^2$, where $p$ is a prime number and $k\geq 2$ is an even integer.

\medskip\noindent
1. Compute Galois conjugacy classes of newforms in $\mathcal{S}_{k}(\Gamma_{0}(p))$. Call the set $New$.

\medskip\noindent
2. Compute the Sturm bound $B=\frac{k}{12}[SL_{2}(\mathbb{Z}):\Gamma_{0}(p)]=\frac{k}{12}\cdot(p+1)$.

\medskip\noindent
3. Compute the coefficients  $a_{n}(E_{k}-p^{k-1}E_{k}^{(p)})$ for $n\leq B$. 

\medskip\noindent
4. Compute the set of primes $L=\{\ell\textrm{ prime}:\ell\mid\textrm{Numerator}(-\frac{B_{k}}{2k}(1-p))\}$.

\medskip\noindent
5. For each pair $(\ell,f)\in L\times New$, compute $K_{f}$, i.e., the coefficient field of $f$. By $f$ we mean here a choice of a representative in Galois conjugacy class.

\medskip\noindent
6. Find an algebraic integer $\theta$ such that $K_{f}=\mathbb{Q}(\theta)$. 

\medskip\noindent
7. Compute an $\ell$-maximal order $\mathcal{O}$ above $\mathbb{Z}[\theta]$.

\medskip\noindent
8. Compute the set $\mathcal{S}=\{\lambda\in\textrm{ Spec }\mathcal{O}:\lambda\cap\mathbb{Z}=\ell\mathbb{Z}\}$.

\medskip\noindent
9. For each $\lambda\in\mathcal{S}$ compute
\[r_{\lambda}=\min_{n\leq B}(ord_{\lambda}(a_{n}(f)-a_{n}(E_{k}-p^{k-1}E_{k}^{(p)}))).\]

\medskip\noindent
\textbf{Output:} If $r_{\lambda}>0$ then we have a congruence
\[a_{n}(f)\equiv a_{n}(E_{k}-p^{k-1}E_{k}^{(p)})\textrm{ mod }(\lambda\mathcal{O}_{f})^{r_{\lambda}}\]
for all $n\geq 0$.

\section{Numerical data}\label{Section_numerical}
In this section we discuss numerical results which were gathered by a repeated use of the algorithm in Section \ref{subsection:Sketch_algorithm}. The levels and ranges we have examined are summarized in Table \ref{tab:Levels}.
\begin{table}[htbp]
  \centering
  \caption{Range of computations}
\[\begin{array}{|c|c|c|c|c|c|c|c|c|c|c|c|}
    \hline
    k              & 2     & 4 & 6    & 8  & 10 & 12  & 14 & 16 & 18 & 20 & 22 \\
    \hline
    p\leq & 3001 & 919 & 251 & 193 & 109 & 113 & 97 & 73 & 71 & 71 & 59\\
    \hline
\end{array}\]
  \label{tab:Levels}%
\end{table}%
\noindent
In total, we found 765 congruences of the form (\ref{congruence}) for the ranges and weights described above. There are 73 congruences such that $r>1$. We found 115 congruences such that $\lambda$ is ramified, i.e. $ord_{\lambda}(\ell) >1 $. Only 7 among them have the property that $r>1$. For large levels, when the degree of the coefficient field $K_{f}$ of a newform $f$ was bigger than $150$ or the prime ideal $\lambda$ contained $2$, we have skipped the computations. However, we did a complete search in the range described in Table \ref{tab:Levels_full}.
\begin{table}[htbp]
  \centering
  \caption{Levels and weights which were completely investigated}
\[\begin{array}{|c|c|c|c|c|c|c|c|c|c|c|c|}
    \hline
    k          & 2     & 4 & 6    & 8  & 10 & 12  & 14 & 16 & 18 & 20 & 22 \\
    \hline
    p\leq & 1789 & 397 & 229 & 193 & 109 & 113 & 97 & 71 & 67 & 67 & 59 \\
    \hline
\end{array}\]
  \label{tab:Levels_full}%
\end{table}%
\subsection{Description of the tables}
\noindent
We are interested in a congruence of the type
\[a_{n}(E)\equiv a_{n}(f)\textrm{ mod }\lambda^{r}\]
for all $n\geq 0$, between the Eisenstein series $E=E_{k}-p^{k-1}E_{k}^{(p)}\in\mathcal{E}_{k}(\Gamma_{0}(p))$ and the newform $f\in\mathcal{S}_{k}(\Gamma_{0}(p))$ for different weights $k$ and prime levels $p$. 

\medskip\noindent
Our numerical results are summarized in Table \ref{tab:UnramifiedHigherPowers}, Table \ref{tab:RamifiedHighPowers} and Table \ref{tab:RamifiedSmallPowers} presented below. Complete data set is available online, cf.\cite{Bartek}. Each table contains $8$ columns.
\begin{table}[h]
  \centering
  \caption{Sample data entry}
    \begin{tabular}{|c|c|c|c|c|c|c|c|}
    \hline
    p     & k     & $\ell$     & r   & m & i     & d \\
    \hline
    769   & 2     & 2     & 5     & 5     & 2     & 36  \\
    \hline
\end{tabular}
  \label{tab:sample_data_entry}%
\end{table}
The letters $p$, $k$ and $\ell$ were already explained. Denote by $d$ the degree of the number field $K_{f}$ generated by the coefficients of the form $f$ and $\lambda$ is a prime ideal in the ring of integers of $K_{f}$, above the rational prime $\ell\in\mathbb{Z}$. The letter $e$ denotes the ramification index of the ideal $\lambda$ at $\ell$ and $m=ord_{\lambda}(\mu(E))$. The column labeled by $i$ contains the number of the Galois orbit of representing newform $f$ (with respect to the internal MAGMA numbering). 

\medskip\noindent
In Table \ref{tab:UnramifiedHigherPowers} we present data concerning congruences for which $ord_{\lambda}(\ell)=1$. In Table $\ref{tab:RamifiedHighPowers}$ we present cases where $ord_{\lambda}(\ell) >1$ and $ord_{\lambda}(\mu(E))>ord_{\lambda}(\ell)$. From Table $\ref{tab:UnramifiedHigherPowers}$ one can read off many properties of the congruences satisfying $ord_{\lambda}(\ell)=1$. For  $1<r\leq ord_{\lambda}(\mu(E))$ we found only 5 congruences that do not satisfy $r = ord_{\lambda}(\mu(E))$ and 56 that satisfy this condition. Observe that the exponent was not maximal only for $k=2$. In Table \ref{tab:RamifiedHighPowers} we collect data about all congruences for which $ord_{\lambda}(\ell)>1$ and $ord_{\lambda}(\mu(E))>ord_{\lambda}(\ell)$. For primes $\ell\geq 3$ we found only $5$. The cases when $ord_{\lambda}(\mu(E))$ equals $ord_{\lambda}(\ell)$ are presented in Table \ref{tab:RamifiedSmallPowers}.

\begin{table}[htbp]
  \centering
  \caption{Congruences with exponent greater than one and without ramification}
\begin{tabular}{ccccc}
    \begin{tabular}{|c|c|c|c|c|c|c|c|}
    \hline
    p     & k     & $\ell$     & r   & m & i     & d \\
    \hline
    769   & 2     & 2     & 5     & 5     & 2     & 36  \\
    \hline
    1459  & 2     & 3     & 5     & 5     & 3     & 71  \\
    \hline
    257   & 4     & 2     & 4     & 4     & 1     & 28  \\
    \hline
    641   & 2     & 2     & 4     & 4     & 2     & 33  \\
    \hline
    1409  & 2     & 2     & 4     & 4     & 3     & 65  \\
    \hline
    163   & 2     & 3     & 3     & 3     & 3     & 7   \\
    \hline
    163   & 4     & 3     & 3     & 3     & 1     & 19  \\
    \hline
    163   & 8     & 3     & 3     & 3     & 1     & 46  \\
    \hline
    193   & 2     & 2     & 3     & 3     & 3     & 8   \\
    \hline
    193   & 6     & 2     & 3     & 3     & 2     & 41  \\
    \hline
    251   & 2     & 5     & 3     & 3     & 2     & 17  \\
    \hline
    449   & 2     & 2     & 3     & 3     & 2     & 23  \\
    \hline
    487   & 2     & 3     & 3     & 4     & 4     & 16  \\
    \hline
    577   & 2     & 2     & 3     & 3     & 4     & 18  \\
    \hline
    811   & 2     & 3     & 3     & 3     & 3     & 40  \\
    \hline
    1373  & 2     & 7     & 3     & 3     & 3     & 60  \\
    \hline
    1601  & 2     & 2     & 3     & 3     & 2     & 80  \\
    \hline
    1783  & 2     & 3     & 3     & 3     & 2     & 82  \\
    \hline
    97    & 2     & 2     & 2     & 2     & 2     & 4   \\
    \hline
    97    & 6     & 2     & 2     & 2     & 2     & 21  \\
    \hline
    97    & 10    & 2     & 2     & 2     & 2     & 37  \\
    \hline
    \end{tabular} & &

    \begin{tabular}{|c|c|c|c|c|c|c|c|}
    \hline
    p     & k     & $\ell$     & r   & m & i     & d  \\
    \hline
    101   & 2     & 5     & 2     & 2     & 2     & 7   \\
    \hline
    101   & 6     & 5     & 2     & 2     & 2     & 24  \\
    \hline
    101   & 10    & 5     & 2     & 2     & 2     & 41  \\
    \hline
    109   & 2     & 3     & 2     & 2     & 3     & 4   \\
    \hline
    109   & 4     & 3     & 2     & 2     & 1     & 12  \\
    \hline
    109   & 8     & 3     & 2     & 2     & 1     & 30  \\
    \hline
    109   & 10    & 3     & 2     & 2     & 2     & 42  \\
    \hline
    151   & 2     & 5     & 2     & 2     & 3     & 6   \\
    \hline
    151   & 6     & 5     & 2     & 2     & 2     & 35  \\
    \hline
    163   & 6     & 3     & 2     & 2     & 2     & 35  \\
    \hline
    193   & 4     & 2     & 2     & 2     & 1     & 23  \\
    \hline
    197   & 2     & 7     & 2     & 2     & 3     & 10  \\
    \hline
    197   & 4     & 7     & 2     & 2     & 1     & 22  \\
    \hline
    251   & 4     & 5     & 2     & 2     & 1     & 24  \\
    \hline
    379   & 2     & 3     & 2     & 2     & 2     & 18  \\
    \hline
    379   & 4     & 3     & 2     & 2     & 1     & 44  \\
    \hline
    433   & 2     & 3     & 2     & 2     & 4     & 16  \\
    \hline
    491   & 2     & 7     & 2     & 2     & 3     & 29  \\
    \hline
    601   & 2     & 5     & 2     & 2     & 2     & 29  \\
    \hline
    673   & 2     & 2     & 2     & 2     & 3     & 24  \\
    \hline
    677   & 2     & 13    & 2     & 2     & 4     & 35  \\
    \hline
    \end{tabular} & &

    \begin{tabular}{|c|c|c|c|c|c|c|c|}
    \hline
    p     & k     & $\ell$     & r   & m & i     & d  \\
    \hline
    727   & 2     & 11    & 2     & 2     & 2     & 36  \\
    \hline
    751   & 2     & 5     & 2     & 3     & 2     & 38  \\
    \hline
    757   & 2     & 3     & 2     & 2     & 2     & 33  \\
    \hline
    883   & 2     & 7     & 2     & 2     & 2     & 39  \\
    \hline
    929   & 2     & 2     & 2     & 2     & 3     & 47  \\
    \hline
    1051  & 2     & 5     & 2     & 2     & 3     & 48  \\
    \hline
    1151  & 2     & 5     & 2     & 2     & 3     & 68  \\
    \hline
    1153  & 2     & 2     & 2     & 4     & 3     & 50  \\
    \hline
    1201  & 2     & 5     & 2     & 2     & 3     & 51  \\
    \hline
    1217  & 2     & 2     & 2     & 3     & 2     & 58  \\
    \hline
    1301  & 2     & 5     & 2     & 2     & 3     & 66  \\
    \hline
    1451  & 2     & 5     & 2     & 2     & 2     & 73  \\
    \hline
    1453  & 2     & 11    & 2     & 2     & 2     & 63  \\
    \hline
    1471  & 2     & 7     & 2     & 2     & 2     & 72  \\
    \hline
    1567  & 2     & 3     & 2     & 2     & 4     & 69  \\
    \hline
    1601  & 2     & 5     & 2     & 2     & 2     & 80  \\
    \hline
    1621  & 2     & 3     & 2     & 3     & 3     & 70  \\
    \hline
    1667  & 2     & 7     & 2     & 2     & 2     & 82  \\
    \hline
    1697  & 2     & 2     & 2     & 2     & 2     & 77  \\
    \hline
    \end{tabular}
    \end{tabular}%
  \label{tab:UnramifiedHigherPowers}%
\end{table}%

\begin{table}[htbp]
  \centering
  \caption{Congruences with $m>e$ and with ramification}
\begin{tabular}{ccccc}
    \begin{tabular}{|c|c|c|c|c|c|c|c|}
    \hline
    p     & k     & $\ell$     & r   & m  & e   & i     & d \\
    \hline
    \bf{3001} & \bf{2}     & \bf{5}     & \bf{1}     & \bf{6}     & \bf{2}     & \bf{1}     & \bf{2} \\
    \hline
    \bf{3001} & \bf{2}     & \bf{5}     & \bf{1}     & \bf{9}     & \bf{3}     & \bf{3}     & \bf{132} \\
    \hline
    \bf{251}   & \bf{6}     & \bf{5}     & \bf{1}     & \bf{6}     & \bf{2}     & \bf{2}     & \bf{59}  \\
    \hline
    \bf{919}   & \bf{2}     & \bf{3}     & \bf{2}     & \bf{4}     & \bf{2}     & \bf{3}     & \bf{47}  \\
    \hline
    \bf{919}   & \bf{4}     & \bf{3}     & \bf{2}     & \bf{4}     & \bf{2}     & \bf{1}     & \bf{105}  \\
    \hline
    257   & 2     & 2     & 1     & 25    & 5     & 2     & 14  \\
    \hline
    257   & 2     & 2     & 5     & 10    & 2     & 2     & 14  \\
    \hline
    353   & 2     & 2     & 1     & 10    & 5     & 4     & 14  \\
    \hline
\end{tabular} & &

    \begin{tabular}{|c|c|c|c|c|c|c|c|}
    \hline
    p     & k     & $\ell$     & r   & m  & e   & i     & d \\
    \hline
    577   & 2     & 2     & 1     & 6     & 2     & 4     & 18  \\
    \hline
    1153  & 2     & 2     & 1     & 16    & 4     & 3     & 50  \\
    \hline
    1249  & 2     & 2     & 1     & 26    & 13    & 3     & 59  \\
    \hline
    1601  & 2     & 2     & 1     & 6     & 2     & 2     & 80  \\
    \hline
    1217  & 2     & 2     & 1     & 39    & 13    & 2     & 58  \\
    \hline
    1889  & 2     & 2     & 1     & 4     & 2     & 3     & 96  \\
    \hline
    2113  & 2     & 2     & 1    & 15     & 5     & 2     & 91  \\
    \hline
    2273  & 2     & 2     & 1    & 10     & 5     & 3     & 105 \\
    \hline
    \end{tabular} & &

    \begin{tabular}{|c|c|c|c|c|c|c|c|}
    \hline
    p     & k     & $\ell$     & r   & m  & e   & i     & d \\
    \hline
    257   & 4     & 2     & 1     & 12    & 3     & 1     & 28  \\
    \hline
    257   & 4     & 2     & 1     & 16    & 4     & 2     & 36  \\
    \hline
    257   & 4     & 2     & 1     & 20    & 5     & 1     & 28  \\
    \hline
    257   & 4     & 2     & 1     & 20    & 5     & 2     & 36  \\
    \hline
    257   & 4     & 2     & 1     & 20    & 5     & 2     & 36  \\
    \hline
    257   & 4     & 2     & 1     & 8     & 2     & 1     & 28  \\
    \hline
    257   & 4     & 2     & 5     & 8     & 2     & 1     & 28  \\
    \hline
    \end{tabular}

\end{tabular}
  \label{tab:RamifiedHighPowers}%
\end{table}%

\begin{table}[htbp]
  \centering
  \caption{Congruences with $m=e$ and with ramification}
\begin{tabular}{ccccc}
    \begin{tabular}{|c|c|c|c|c|c|c|c|}
    \hline
    p     & k     & $\ell$     & r     & m     & e     & i     & d  \\
    \hline
    31    & 2     & 5     & 1     & 2     & 2     & 1     & 2   \\
    \hline
    31    & 6     & 5     & 1     & 2     & 2     & 2     & 8   \\
    \hline
    31    & 10    & 5     & 1     & 2     & 2     & 2     & 13  \\
    \hline
    31    & 14    & 5     & 1     & 2     & 2     & 2     & 18  \\
    \hline
    31    & 18    & 5     & 1     & 2     & 2     & 2     & 23  \\
    \hline
    31    & 22    & 5     & 1     & 2     & 2     & 2     & 28  \\
    \hline
    47    & 10    & 23    & 1     & 2     & 2     & 2     & 20  \\
    \hline
    47    & 12    & 23    & 1     & 2     & 2     & 1     & 18  \\
    \hline
    47    & 16    & 23    & 1     & 2     & 2     & 1     & 26  \\
    \hline
    47    & 20    & 23    & 1     & 2     & 2     & 1     & 34  \\
    \hline
    53    & 6     & 13    & 1     & 2     & 2     & 2     & 12  \\
    \hline
    53    & 18    & 13    & 1     & 2     & 2     & 2     & 38  \\
    \hline
    67    & 4     & 11    & 1     & 2     & 2     & 1     & 7   \\
    \hline
    67    & 14    & 11    & 1     & 2     & 2     & 2     & 37  \\
    \hline
    103   & 2     & 17    & 1     & 2     & 2     & 2     & 6   \\
    \hline
    113   & 2     & 2     & 1     & 2     & 2     & 2     & 2   \\
    \hline
    113   & 6     & 2     & 1     & 2     & 2     & 1     & 21  \\
    \hline
    113   & 6     & 2     & 1     & 2     & 2     & 1     & 21  \\
    \hline
    113   & 6     & 2     & 1     & 2     & 2     & 2     & 25  \\
    \hline
    113   & 6     & 2     & 1     & 2     & 2     & 2     & 25  \\
    \hline
    127   & 2     & 7     & 1     & 2     & 2     & 2     & 7   \\
    \hline
    127   & 8     & 7     & 1     & 2     & 2     & 1     & 34  \\
    \hline
    131   & 2     & 5     & 1     & 2     & 2     & 2     & 10  \\
    \hline
    131   & 6     & 5     & 1     & 2     & 2     & 2     & 32  \\
    \hline
    191   & 6     & 5     & 1     & 2     & 2     & 2     & 46  \\
    \hline
    199   & 2     & 3     & 1     & 2     & 2     & 3     & 10  \\
    \hline
    199   & 4     & 3     & 1     & 2     & 2     & 1     & 20  \\
    \hline
    211   & 2     & 5     & 1     & 2     & 2     & 1     & 2   \\
    \hline
    211   & 6     & 5     & 1     & 2     & 2     & 2     & 47  \\
    \hline
    223   & 4     & 37    & 1     & 2     & 2     & 1     & 24  \\
    \hline
    \end{tabular} & &

   \begin{tabular}{|c|c|c|c|c|c|c|c|}
    \hline
    p     & k     & $\ell$     & r     & m     & e     & i     & d  \\
    \hline
    281   & 2     & 5     & 1     & 2     & 2     & 2     & 16  \\
    \hline
    307   & 4     & 3     & 1     & 2     & 2     & 1     & 35  \\
    \hline
    337   & 2     & 2     & 1     & 2     & 2     & 2     & 15  \\
    \hline
    337   & 4     & 7     & 1     & 2     & 2     & 1     & 40  \\
    \hline
    353   & 4     & 2     & 1     & 2     & 2     & 2     & 48  \\
    \hline
    353   & 4     & 11    & 1     & 2     & 2     & 1     & 40  \\
    \hline
    367   & 4     & 61    & 1     & 2     & 2     & 1     & 41  \\
    \hline
    401   & 4     & 5     & 1     & 2     & 2     & 1     & 45  \\
    \hline
    409   & 2     & 17    & 1     & 2     & 2     & 2     & 20  \\
    \hline
    409   & 4     & 17    & 1     & 2     & 2     & 1     & 47  \\
    \hline
    419   & 4     & 19    & 1     & 2     & 2     & 1     & 43  \\
    \hline
    523   & 2     & 3     & 1     & 2     & 2     & 3     & 26  \\
    \hline
    541   & 2     & 5     & 1     & 2     & 2     & 2     & 24  \\
    \hline
    571   & 2     & 5     & 1     & 2     & 2     & 9     & 18  \\
    \hline
    593   & 2     & 2     & 1     & 2     & 2     & 5     & 27  \\
    \hline
    661   & 2     & 11    & 1     & 2     & 2     & 3     & 29  \\
    \hline
    683   & 2     & 11    & 1     & 2     & 2     & 3     & 31  \\
    \hline
    691   & 2     & 5     & 1     & 2     & 2     & 2     & 33  \\
    \hline
    733   & 2     & 61    & 1     & 2     & 2     & 4     & 32  \\
    \hline
    761   & 2     & 5     & 1     & 2     & 2     & 3     & 41  \\
    \hline
    761   & 2     & 19    & 1     & 2     & 2     & 3     & 41  \\
    \hline
    881   & 2     & 2     & 1     & 2     & 2     & 2     & 46  \\
    \hline
    911   & 2     & 7     & 1     & 2     & 2     & 3     & 53  \\
    \hline
    941   & 2     & 5     & 1     & 2     & 2     & 2     & 50  \\
    \hline
    971   & 2     & 5     & 1     & 2     & 2     & 2     & 55  \\
    \hline
    1021  & 2     & 17    & 1     & 2     & 2     & 2     & 47  \\
    \hline
    1091  & 2     & 5     & 1     & 2     & 2     & 3     & 62  \\
    \hline
    1201  & 2     & 2     & 1     & 2     & 2     & 1     & 2   \\
    \hline
    1279  & 2     & 3     & 1     & 2     & 2     & 2     & 64  \\
    \hline
    1289  & 2     & 7     & 1     & 2     & 2     & 4     & 61  \\
    \hline
    \end{tabular} & &

   \begin{tabular}{|c|c|c|c|c|c|c|c|}
    \hline
    p     & k     & $\ell$     & r     & m     & e     & i     & d  \\
    \hline
    1291  & 2     & 5     & 1     & 2     & 2     & 2     & 62  \\
    \hline
    1381  & 2     & 5     & 1     & 2     & 2     & 2     & 63  \\
    \hline
    1447  & 2     & 241   & 1     & 2     & 2     & 2     & 71  \\
    \hline
    1471  & 2     & 5     & 1     & 2     & 2     & 2     & 72  \\
    \hline
    1483  & 2     & 13    & 1     & 2     & 2     & 4     & 67  \\
    \hline
    1511  & 2     & 5     & 1     & 2     & 2     & 2     & 87  \\
    \hline
    1531  & 2     & 3     & 1     & 2     & 2     & 4     & 73  \\
    \hline
    1531  & 2     & 5     & 1     & 2     & 2     & 4     & 73  \\
    \hline
    1553  & 2     & 2     & 1     & 2     & 2     & 2     & 74  \\
    \hline
    1693  & 2     & 3     & 1     & 2     & 2     & 3     & 72  \\
    \hline
    1697  & 2     & 53    & 1     & 2     & 2     & 2     & 77  \\
    \hline
    1777  & 2     & 2     & 1     & 2     & 2     & 2     & 79  \\
    \hline
    1789  & 2     & 149   & 1     & 2     & 2     & 2     & 80  \\
    \hline
    101   & 4     & 5     & 1     & 3     & 3     & 1     & 9   \\
    \hline
    101   & 8     & 5     & 1     & 3     & 3     & 1     & 26  \\
    \hline
    101   & 12    & 5     & 1     & 3     & 3     & 1     & 42  \\
    \hline
    181   & 2     & 5     & 1     & 3     & 3     & 2     & 9   \\
    \hline
    181   & 6     & 5     & 1     & 3     & 3     & 2     & 40  \\
    \hline
    353   & 4     & 2     & 1     & 3     & 3     & 1     & 40  \\
    \hline
    1321  & 2     & 11    & 1     & 3     & 3     & 4     & 56  \\
    \hline
    1381  & 2     & 23    & 1     & 3     & 3     & 2     & 63  \\
    \hline
    1571  & 2     & 5     & 1     & 3     & 3     & 2     & 82  \\
    \hline
    1747  & 2     & 3     & 1     & 3     & 3     & 3     & 77  \\
    \hline
    1201  & 2     & 2     & 1     & 5     & 5     & 3     & 51  \\
    \hline
    353   & 4     & 2     & 1     & 5     & 5     & 1     & 40  \\
    \hline
    353   & 4     & 2     & 1     & 5     & 5     & 2     & 48  \\
    \hline
    353   & 4     & 2     & 1     & 5     & 5     & 2     & 48  \\
    \hline
    353   & 4     & 2     & 2     & 2     & 2     & 1     & 40  \\
    \hline
    43    & 8     & 7     & 2     & 2     & 2     & 1     & 11  \\
    \hline
    43    & 20    & 7     & 2     & 2     & 2     & 1     & 32  \\
    \hline
    \end{tabular}

\end{tabular}
\label{tab:RamifiedSmallPowers}%
\end{table}%

\begin{remark}
The main difficulty in enlarging the number of congruences in Table \ref{tab:RamifiedHighPowers} lies in the fact that $\dim\mathcal{S}_{k}(\Gamma_{0}(p))=O(k\cdot p)$ as $k\rightarrow\infty$ and $p\rightarrow\infty$. In a typical situation, when $f\in\mathcal{S}_{k}(\Gamma_{0}(p))$ is a newform, the degree $[K_{f}:\mathbb{Q}]$ is roughly of the size $\frac{1}{2}\dim\mathcal{S}_{k}(\Gamma_{0}(p))$. To check the congruence, we perform Step 9 of the algorithm in Section \ref{subsection:Sketch_algorithm}. We have to compute $\frac{k}{12}(p+1)$ coefficients of the newform $f$ and this is usually the slowest part of the algorithm. For example, when $k=2$ and $p>3000$, this means that we work with the field $K_{f}$ of degree at least 150 over $\mathbb{Q}$. In the range described in Table \ref{tab:Levels_full} we found all possible congruences such that $ord_{\lambda}(\ell)>1$. For levels and weights described in Table \ref{tab:Levels}, we have decided to look only for congruences such that $[K_{f}:\mathbb{Q}]<150$. This condition guarantees that Step $9$ of the algorithm can be executed in less than $48$ hours of computational time on the computer with Intel i5, 2.53 GHz processor and 4GB RAM.  
\end{remark}


\subsection{The case $r=m>1$ and $e=1$}\label{subsection:r_eq_m_gt_1_case}
Let $k=2$ and $p=109$. In this example we choose any root $\alpha\in\overline{\mathbb{Q}}$ of the equation
\[\alpha^4+\alpha^3-5\alpha^2-4\alpha+3=0\]
and form $K=\mathbb{Q}(\alpha)$. We have the Galois conjugacy class of newforms with the $q$-expansion
\[f=q+\alpha q^2+(1+4\alpha-\alpha^3)q^3+(\alpha^2-2)q^4-\alpha q^5+\ldots.\]
The ring of integers $\mathcal{O}_{f}$ of $K_{f}=K$ is equal to $\mathbb{Z}[\alpha]$ and
\[(3)=(3,\alpha)(3,2+\alpha+\alpha^2+\alpha^3)\]
is the factorization into prime ideals in $\mathcal{O}_{f}$. 
We find, by the algorithm, that for $\lambda=(3,\alpha)$
\[a_{n}(f)\equiv a_{n}(E_{2}-109E_{2}^{(109)})\textrm{ mod }\lambda^2\]
for all natural $n\geq 0$. In fact, this is the maximal possible exponent, since $\mu(E_{2}-109E_{2}^{(109)})=\frac{9}{2}$ and $\textrm{ord}_{\lambda}(9)=2$. In the unramified case, the upper bound for the maximal exponent $r$ is smaller or equal to the one described in Corollary \ref{corollary:bounds}. This example shows that the equality can occur.

\medskip\noindent
If $p=163$ we obtain four different congruences for weights $k=2,4,6$ and $8$ with ideals above $3$ raised to the powers $3,3,2$ and $3$ respectively. For weights $k=2,4$ or $8$ the exponent of the ideal is maximal possible (cf. Table \ref{tab:UnramifiedHigherPowers}). For $k=2$ we find a number field of degree $7$ over $\mathbb{Q}$ with a primitive element $\alpha$ with a minimal polynomial
\[6+4 \alpha -23 \alpha ^2+19 \alpha ^4-5 \alpha ^5-3 \alpha ^6+\alpha ^7=0.\]
The ring of integers is equal to $\mathbb{Z}[\alpha]$. Its discriminant is equal to $2\cdot 82536739$ and 
\[3\mathbb{Z}[\alpha]=(3,\alpha)(3,1+\alpha +\alpha ^3+\alpha ^4+\alpha ^6).\]
We find a newform of level $163$ and weight $2$ with $q$-expansion
\begin{eqnarray*}
f=q + \alpha q^2 + (-2 + 5 \alpha + 5 \alpha^2 - 6 \alpha^3 - \alpha^4 + \alpha^5) q^3 \\
+ (-2 +\alpha^2) q^4 + (6 + 6 \alpha - 11 \alpha^2 - 6 \alpha^3 + 7 \alpha^4 + \alpha^5 - \alpha^6) q^5+\ldots
\end{eqnarray*}
It is congruent to the Eisenstein series
\[E_{2}-163E_{2}^{(163)}=\frac{27}{4}+\sum_{n=1}^{\infty}\sigma_{1}(n)q^n-163\sum_{n=1}^{\infty}\sigma_{1}(n)q^{163n}\]
modulo $(3,\alpha)^3$.

\subsection{The case $m>r>1$ and $e=1$}\label{subsection:m_gt_r_gt_1_e_eq_1_case}
Let $k=2$ and $p=487$. The space of cusp forms $\mathcal{S}_{2}(\Gamma_{0}(487))$ contains five Galois conjugacy classes of newforms. We take $f$ such that
\[f=q+\alpha q+\cdots,\]
where $\alpha$ is an algebraic integer such that
\begin{dmath*}
\alpha^{16} - 7 \alpha^{15} - 5 \alpha^{14} + 131 \alpha^{13} - 132 \alpha^{12} - 977 \alpha^{11} + 1666 \alpha^{10} + 3671 \alpha^9 - 8191 \alpha^8 - 7212 \alpha^7 + 20571 \alpha^6 + 6937 \alpha^5 - 27100 \alpha^4 - 2748 \alpha^3 + 17207 \alpha^2 + 360 \alpha - 3825 =0.
\end{dmath*}
The field $K_{f}$ is equal to $\mathbb{Q}(\alpha)$ and the ideal $3\mathcal{O}_{K}$ is a product of four distinct prime ideals
\[3\mathcal{O}_{K}=\lambda_{1}\lambda_{2}\lambda_{3}\lambda_{4}.\]
Let $\lambda_{1}=(3,\frac{1}{105}\beta)$, where
\begin{dmath*}
\beta=2 \alpha^{15} + 106 \alpha^{14} + 50 \alpha^{13} + 112 \alpha^{12} + 156 \alpha^{11} + 161 \alpha^{10} + 392 \alpha^9 + 307 \alpha^8 + 148 \alpha^7 + 126 \alpha^6 + 192 \alpha^5 + 194 \alpha^4 + 280 \alpha^3 + 279 \alpha^2 + 124 \alpha + 705.
\end{dmath*}
We check that $f$ is congruent to $E_{2}-487E_{2}^{(487)}$ modulo $\lambda_{1}^{3}$ and $m=\mu(E_{2}-487E_{2}^{(487)})=4$ (cf. Corollary \ref{corollary:bounds}). Moreover, there is no congruence modulo $\lambda_{1}^{4}$, hence the maximal exponent $r=3$ is smaller than the theoretical upper bound $m$ in Corollary \ref{corollary:bounds}. 



\subsection{The case $r<e<m$ and $e>1$}\label{subsection:m_gt_e_gt_r_e_gt_1_case}
Let $k=2$ and $p=3001$. The space of cusp forms $\mathcal{S}_{2}(\Gamma_{0}(3001))$ has dimension $249$ and it is a direct sum of three subspaces $S_{1}$, $S_{2}$ and $S_{3}$ of dimensions $2$, $115$ and $132$, respectively. The space $S_{1}$ is generated by two Galois conjugate newforms
\[f_{1}=q + \alpha_{1} q^2 + (\alpha_{1} + 1)q^3 + (\alpha_{1} - 1)q^4 + 2\alpha_{1} q^5 + (2\alpha_{1} + 1)q^6+\cdots,\]
\[f_{2}=q + \alpha_{2} q^2 + (\alpha_{2} + 1)q^3 + (\alpha_{2} - 1)q^4 + 2\alpha_{2} q^5 + (2\alpha_{2} + 1)q^6+\cdots,\]
where $\alpha_{1}$ and $\alpha_{2}$ are roots of the polynomial $x^2-x-1$. Since the forms are Galois conjugate, we will consider only one of them. Assume $\alpha_{1}=\frac{1+\sqrt{5}}{2}$. The ring of integers of $K_{f_{1}}=\mathbb{Q}(\sqrt{5})$ is $\mathcal{O}_{f_{1}}=\mathbb{Z}[\frac{1+\sqrt{5}}{2}]$ and
\[5\mathcal{O}_{f_{1}}=\lambda^{2},\]
for the prime ideal $\lambda$ which equals $(5,2+\frac{1+\sqrt{5}}{2})$. We check that $a_{0}(E_{2}-3001E_{2}^{(3001)})=125$ and $\mu(E_{2}-3001E_{2}^{(3001)})=125$, and $ord_{\lambda}(125)=6$. Corollary \ref{corollary:bounds} shows that the upper bound for the exponent $r$ of the congruence is $6$. We checked by MAGMA that for $n\leq \frac{3001+1}{12}$ the congruence
\[a_{n}(f_{1})\equiv a_{n}(E_{2}-3001E_{2}^{(3001)})\textrm{ mod }\lambda\]
holds. Hence, by Theorem \ref{Sturm_gen} the congruence holds for all $n\geq 0$. But we also find that $a_{2}(f_{1})-a_{2}(E_{2}-3001E_{2}^{(3001)})=\frac{1+\sqrt{5}}{2}-3\notin\lambda^{2}$, which proves that the maximal exponent $r$, for which the congruence holds is equal to $1$.

\subsection{The case $r=e<m$ and $e>1$}\label{subsection:m_gt_e_eq_r_e_gt_1_case}
Let $k=2$ and $p=919$. The space $S_{2}(\Gamma_{0}(919))$ contains 3 Galois conjugacy classes of newforms, with coefficient fields of degrees 2, 27 and 47, respectively. We take a representative $f$ of the class with the coefficient field of degree $47$ over $\mathbb{Q}$. The form $f$ equals $q+\alpha q+\ldots$, where the algebraic integer $\alpha$ is a root of a monic integral polynomial of degree 47. The discriminant of the field $K_{f}$ is approximately equal to $0.5995\cdot 10^{304}$. We were not able to compute the factorization of the discriminant. Instead, we work with a $3$-maximal order $\mathcal{O}$ above $\mathbb{Z}[\alpha]$, where $K_{f}=\mathbb{Q}(\alpha)$. There are $6$ different prime ideals above $3\mathcal{O}$ and $3\mathcal{O}=\lambda_{1}^{2}\cdot\prod_{i=2}^{6}\lambda_{i}$. We check that $f$ is congruent to $E_{2}-919E_{2}^{(919)}$ modulo $\lambda_{1}^2$, while the maximal exponent $m=ord_{\lambda_{1}}(\mu(E_{2}-919E_{2}^{(919)}))$ equals $4$. The discussion in Section \ref{Subsection_orders} implies that there is a congruence between $f$ and $E_{2}-919E_{2}^{(919)}$ modulo $(\lambda_{1}O_{f})^{2}$ and $ord_{\lambda_{1}O_{f}}(\mu(E_{2}-919E_{2}^{(919)}))=4$, hence we find that $r=e<m$ and $e>1$.

\subsection{Equation order is not always $\ell$-maximal}\label{subsection:Equation_order_is_not_always}
It is not always true that if we have a congruence modulo a power of a prime ideal above $\ell$ and $K_{f}=\mathbb{Q}(\theta)$, where $\theta$ is an algebraic integer, then an $\ell$-maximal order above $\mathbb{Z}[\theta]$ that we get from the algorithm implemented in MAGMA, is equal to the ring $\mathbb{Z}[\theta]$. We summarize several examples in Table \ref{tab:IndexHigh}. The prime $\ell$ is unramified in $K_{f}$. By $i$ we denote the number of the Galois orbit of the newform and by $ind$ the index $[\mathcal{O}:\mathbb{Z}[\theta]]$ for the $\ell$-maximal order above $\mathbb{Z}[\theta]$.
\begin{table}[htbp]
  \centering
  \caption{Index of the order}
    \begin{tabular}{|c|c|c|c|c|}
    \hline
    p     & k    & $\ell$   & i  & ind   \\
    \hline
    101   & 6    & 5      & 2  & 625 \\
    \hline
    751   & 2    & 5      & 2  & 625 \\
    \hline
    1621   & 2    & 3      & 3  & 3 \\
    \hline
    1667   & 2    & 7      & 2  & 343 \\
    \hline
 \end{tabular}
  \label{tab:IndexHigh}
\end{table}

\subsection{Large residue field}\label{subsection:large_residue_field}
Let $k=2$ and $p=401$. The space $\mathcal{S}_{2}(\Gamma_{0}(401))$ is a direct sum $S_{1}\oplus S_{2}$ of two new subspaces. The space $S_{1}$ is of dimension $12$ over $\mathbb{C}$ and it is generated by a newform and its Galois conjugates, none of which is congruent to the Eisenstein series $E_{2}-401E_{2}^{(401)}$ modulo any power of a prime ideal. The space $S_{2}$ is generated by a newform $f=q+\alpha q^2+\cdots$, such that $\alpha$ is an algebraic integer satisfying
\begin{dmath*}
-44+1058 \alpha-4111 \alpha^2-24699 \alpha^3+12831 \alpha^4+93934 \alpha^5-14353 \alpha^6-152221 \alpha^7+8292 \alpha^8+132085 \alpha^9-2749 \alpha^{10}-67876 \alpha^{11}+519 \alpha^{12}+21617 \alpha^{13}-51 \alpha^{14}-4305 \alpha^{15}+2 \alpha^{16}+521 \alpha^{17}-35 \alpha^{19}+\alpha^{21}=0.
\end{dmath*}
The field $K_{f}$ is generated by $\theta=\alpha$. The ideal $5\mathcal{O}_{f}$ is a product of four distinct prime ideals, i.e. every prime ideal above $5$ is unramified in $\mathcal{O}_{f}$. The prime ideal $\lambda=(5,\frac{1}{8}\beta)$ with 
\begin{dmath*}
\beta=32+48 \alpha+82 \alpha^2+75 \alpha^3+66 \alpha^4+70 \alpha^5+39 \alpha^6+62 \alpha^7+50 \alpha^8+37 \alpha^9+22 \alpha^{10}+56 \alpha^{11}+17 \alpha^{12}+2 \alpha^{13}+16 \alpha^{14}+26 \alpha^{15}+3 \alpha^{16}+7 \alpha^{17}+\alpha^{18}+\alpha^{19}
\end{dmath*} 
gives the quotient map $\pi:\mathcal{O}_{f}\rightarrow \mathcal{O}_{f}/\lambda\cong \mathbb{F}_{25}$. The image $\pi(\mathbb{Z}[\{a_{n}(f)\}_{n\in\mathbb{N}}\}])$ is a subfield $\mathbb{F}_{5}$ in $\mathcal{O}_{f}/\lambda$. Observe that $a_{0}(E_{2}-401E_{2}^{(401)})=\frac{50}{3}$ and $v_{5}(\mu(E_{2}-401E_{2}^{(401)}))=2$. One could expect that
\[a_{n}(f)\equiv a_{n}(E_{2}-401E_{2}^{(401)})\textrm{ mod }\lambda^{r}\]
holds for all $n\geq 0$ and $r\leq 2$. However, the coefficient $a_{2}(f)-a_{2}(E_{2}-401E_{2}^{(401)})$ is not congruent to $0$ modulo $\lambda^{2}$. The differences $a_{n}(f)- a_{n}(E_{2}-401E_{2}^{(401)})$ are congruent to $0$ modulo $\lambda$ for $n$ less or equal to the Sturm bound, hence, the congruence modulo $\lambda$ holds for all $n\geq 0$.
The order $\mathcal{O}=\mathbb{Z}[\{a_{n}(f)\}_{n\in\mathbb{N}}]$ equals also $\mathbb{Z}[\{a_{2}(f),a_{3}(f), a_{5}(f)\}]$ and can be obtained by performing the \textit{pMaximalOrder} algorithm in MAGMA, starting with $\mathbb{Z}[\theta]$. Let $\mathcal{O}^{(2)}$ denote a $2$-maximal order obtained from $\mathbb{Z}[\theta]$. Then, let $\mathcal{O}^{(2,3697)}$ be an $3697$-maximal order above $\mathcal{O}^{(2)}$ obtained by the MAGMA algorithm. Finally, let $\mathcal{O}^{(2,3697,34759357)}$ denote a $34759357$-maximal order above $\mathcal{O}^{(2,3697)}$. The order $\mathcal{O}^{(2,3697,34759357)}$ equals $\mathcal{O}$ and $[\mathcal{O}_{f}:\mathcal{O}]=5$. The maximal order $\mathcal{O}_{f}$ is the $5$-maximal order above $\mathcal{O}$. Note also that we have proper inclusions
\[\mathbb{Z}[\theta]\subsetneq\mathcal{O}^{(2)}\subsetneq\mathcal{O}^{(2,3697)}\subsetneq\mathcal{O}^{(2,3697,34759357)}.\]
In the algorithm in Section \ref{subsection:Sketch_algorithm} we use only a $5$-maximal order above $\mathbb{Z}[\theta]$. This is sufficient when we check the congruence modulo a prime above $5$ in $\mathcal{O}_{f}$. In order to compute $\mathcal{O}_{f}$, we should run \textit{pMaximalOrder} algorithm with primes $2,5,3697$ and $34759357$, respectively. Observe that in this particular case, the computation of $3697$ and $34759357$-maximal orders is completely unnecessary, because we check congruences modulo primes above $2$ and $5$ and $a_{0}(E_{2}-401E_{2}^{(401)})=\frac{50}{3}$. This shows that in Step 7 of the algorithm in Section \ref{subsection:Sketch_algorithm} we gain some significant amount of time by skipping the superfluous computation of $\mathcal{O}_{f}$.

\subsection{Congruences over $\mathbb{Q}$}\label{subsection:congruences_over_Q}
Let $f$ in $S_{2}(\Gamma_{0}(p))$ be a rational newform for which the congruence (\ref{congruence}) holds for all $n\geq 0$. In particular, $a_{q}(f)\equiv a_{q}(E)=q^{k-1}+1\textrm{ mod }\lambda^{r}$ for all primes $q\neq p$. Finding a rational newform $f$ as above amounts to a search for an elliptic curve $F$ (attached to $f$ by the modularity theorem) defined over $\mathbb{Q}$, of prime conductor $p$, such that
\[|\tilde{F}_{q}(\mathbb{F}_{q})|\equiv 0 \textrm{ mod }\ell^{r},\]
where $\tilde{F}_{q}$ denotes the reduction of $F$ at the prime $q$.
It follows from \cite[Theorem 2]{Katz}, that there exists an elliptic curve $F'$ over $\mathbb{Q}$ which is $\mathbb{Q}$-isogenous to $F$ and the group of $\mathbb{Q}$-rational points on $F'$ contains a point of order $\ell^{r}$. The conductor of $F'$ is $p$. For an elliptic curve defined over $\mathbb{Q}$ the smallest possible conductor is $11$. Hence $F'$ has good reduction at $2$. The group of $\mathbb{Q}$-rational points of $F'$ contains the torsion subgroup, which we denote by $T$. The reduction at 2 maps $T$ into $\tilde{F}'_{2}(\mathbb{F}_{2})$. The kernel of this homomorphism has order a power of 2. By Hasse theorem for elliptic curves the inequality $|\tilde{F}'_{2}(\mathbb{F}_{2})|\leq 5$ holds. Hence, the order $|T|$ equals $2^{m}$, $2^{m}\cdot 3$ or $2^{m}\cdot 5$, for some $m\geq 0$. Suppose that $|T|>2$. 

\medskip\noindent
If $|T|=2^{m}$, then \cite[Theorem 2]{Miyawaki} and \cite[Theorem 3]{Miyawaki} imply that $T\cong \mathbb{Z}/2\mathbb{Z}\oplus\mathbb{Z}/2\mathbb{Z}$ or $T\cong\mathbb{Z}/4\mathbb{Z}$ and in both cases $p=17$. The only $\mathbb{Q}$-isogeny class of elliptic curves of conductor $17$ is attached to the newform $f=q - q^2 - q^4 +\cdots$ in $\mathcal{S}_{2}(\Gamma_{0}(17))$. For any prime $q\neq p$ the coefficient $a_{q}(f)=1+q-|\tilde{F}'_{q}(\mathbb{F}_{q})|$ is congruent to $1+q$ modulo $4$. We check directly that $a_{p}(f)=1$. The congruence $a_{n}(f)\equiv a_{n}(E_{2}-17E_{2}^{(17)})\textrm{ mod }4$ holds for all $n\geq 1$. However, $a_{0}(E_{2}-17E_{2}^{(17)})=\frac{2}{3}$, so for all $n\geq 0$ we have only a congruence modulo $2$. 

\medskip\noindent
If $|T|=2^{m}\cdot 3$, then \cite[Theorem 1]{Miyawaki} implies that $T\cong\mathbb{Z}/3\mathbb{Z}$ and $p=19$ or $p=37$. There is exactly one $\mathbb{Q}$-isogeny class of elliptic curves of conductor 19. It provides the newform $f=q - 2q^3 - 2q^4+\cdots$ in $\mathcal{S}_{2}(\Gamma_{0}(19))$ congruent to $E_{2}-19E_{2}^{(19)}$ modulo 3 at all coefficients. If the conductor $p$ equals $37$, then there are two $\mathbb{Q}$-isogeny classes of elliptic curves. Only the class associated with the newform $f=q+q^3-2q^4+\cdots$ in $\mathcal{S}_{2}(\Gamma_{0}(37))$ provides the congruence $a_{n}(f)\equiv a_{n}(E_{2}-37E_{2}^{(37)})\textrm{ mod }3$ for all $n\geq 0$. The other newform $f'$ satisfies $a_{p}(f')=-1$, so the congruence cannot hold. 

\medskip\noindent
If $|T|=2^{m}\cdot 5$, then by \cite[Theorem 4]{Miyawaki} we get $T\cong\mathbb{Z}/5\mathbb{Z}$ and $p=11$. The unique newform $f\in\mathcal{S}_{2}(\Gamma_{0}(11))$ satisfies $a_{p}(f)=1$. Moreover, $a_{0}(E_{2}-11E_{2}^{(11)})=\frac{5}{12}$, hence the congruence $a_{n}(f)\equiv a_{n}(E_{2}-11E_{2}^{(11)})\textrm{ mod }5$ holds for all $n\geq 0$. 

\medskip\noindent
We are left with only one case, when $|T|$ equals $2$. If the elliptic curve $F'$ of prime conductor $p>17$ satisfies $|T|=2$, then $p=u^2+64$ for some rational number $u$. This is proved in \cite[Theorem 2]{Setzer}. The primes $p=113,353,593$ and $1153$ are of the form $u^2+64$ for $u\in\mathbb{Z}$ and on these levels we find newforms $f\in\mathcal{S}_{2}(\Gamma_{0}(p))$ congruent congruent to $E_{2}-pE_{2}^{(p)}$ modulo 2. However, in general it is not known whether the sequence $\{u^2+64\}_{u\in\mathbb{N}}$ contains an infinite number of primes.



\section*{Acknowledgments}
The author would like to thank Wojciech Gajda for many helpful suggestions and corrections. He thanks Gerhard Frey for reading an earlier version of the paper and for helpful comments and remarks. He would like to thank Gabor Wiese for his help in improving the paper and suggesting one of the lemmas. Finally, the author wishes to express his thanks to an anonymous referee for the careful reading of the paper and a detailed list of comments which improved the exposition and removed several inaccuracies. The author was supported by the National Science Centre research grant 2012/05/N/ST1/02871.

\bibliography{bibliography}
\bibliographystyle{amsplain}

\end{document}